\documentclass[12pt]{amsart}
\usepackage{amscd,amssymb,epsfig,mathtools}
\calclayout

\numberwithin{equation}{section} 

\newcommand{\ud}{\,d} 
 
\newcommand{\R}{\mathbb{R}}

\newcommand\cof{\operatorname{cof}}

\renewcommand{\div}{\operatorname{div}}

\newcommand{\tir}[1]{\ensuremath{\overline {#1}}} 
\newtheorem{thm}{Theorem}[section] 
 
\newtheorem{lemma}[thm]{Lemma}

\newtheorem{rem}[thm]{Remark}

\newtheorem{ass}[thm]{Assumption}

\def\whsq{\vbox to 5.8pt 
{\offinterlineskip\hrule 
\hbox to 5.8pt{\vrule height 
5.1pt\hss\vrule height 5.1pt}\hrule}}

\def\<{\langle} 
\def\>{\rangle} 
\def\PP{{\mathop{{\rm I}\kern-.2em{\rm P}}\nolimits}} 
\def\FF{{\mathop{{\rm I}\kern-.2em{\rm F}}\nolimits}}   
\def\ZZ{{\mathop{{\rm I}\kern-.2em{\rm Z}}\nolimits}} 
 
\newlength{\sidemargin} 
\begin{document}
\title[]{
Standard finite elements for the numerical resolution of the elliptic Monge-Amp\`ere equation: classical solutions}

\thanks{This work began when the author was supported in part by a 2009-2013 Sloan Foundation Fellowship. }
\author{Gerard Awanou}
\address{Department of Mathematics, Statistics, and Computer Science, M/C 249.
University of Illinois at Chicago, 
Chicago, IL 60607-7045, USA}
\email{awanou@uic.edu}  
\urladdr{http://www.math.uic.edu/\~{}awanou}

\maketitle

\begin{abstract}
We propose a new variational formulation of the elliptic Monge-Amp\`ere equation and show how classical Lagrange elements can be used for the numerical resolution of classical solutions of the equation.  Error estimates are given for Lagrange elements of degree $d \geq 2$ in dimensions 2 and 3. No jump term is used in the variational formulation.  
We propose to solve the discrete nonlinear system of equations by a time marching method  and numerical evidence is given which indicates that  one approximates weak solutions in two dimensions. 


\end{abstract}

\section{Introduction}
This paper addresses the numerical resolution of the Dirichlet problem for the Monge-Amp\`ere 
equation
\begin{equation}
\det  D^2 u = f  \ \text{in} \ \Omega, \quad u=g \ \text{on} \ \partial \Omega. 
\label{m1}
\end{equation}
 A classical solution of \eqref{m1} is a convex function $u \in C(\tir{\Omega}) \cap C^2 (\Omega)$ which satisfies \eqref{m1}. The domain $\Omega \subset \mathbb{R}^n, n=2,3$ is assumed to be convex with (polygonal) 
boundary $\partial \Omega$. 
Here $D^2 u=\bigg( \frac{\partial^2 u}{\partial x_i \partial x_j}\bigg)_{i,j=1,\ldots, n}  $ 
is the Hessian of $u$
and $f,g$ are given functions with $f \geq 0$ and $g \in C(\partial \Omega)$ with $g$ convex on any line segment in $\partial \Omega$.  A smooth solution of \eqref{m1} solves the variational problem: find $u \in W^{2,\infty}(\Omega)$ such that $u=g \ \text{on} \ \partial \Omega$ and $$\int_{\Omega}(\det D^2 u) v \ud x = 
\int_{\Omega} f v \ud x, \, \text{for all} \, v \in H^1_0(\Omega).$$

We propose to solve numerically \eqref{m1} with standard Lagrange finite element spaces $V_h$ of degree $d \geq 2$ by analyzing the (nonconforming) variational problem: find $u_h \in V_h \subset H^1(\Omega)$ such that $u_h=g_h \ \text{on} \ \partial \Omega$ for an interpolant $g_h$ of $g$ and
\begin{align} \label{m1h}
\sum_{K \in \mathcal{T}_h } \int_{K}  (\det D^2 u_h) v_h \ud x = 
\int_{\Omega} f v_h \ud x, \forall v_h \in V_h \cap H^1_0(\Omega).
\end{align}
Here $\mathcal{T}_h$ denotes a quasi-uniform, simplicial and conforming triangulation of the domain. Error estimates for smooth solutions are derived. 
We propose to solve the discrete nonlinear system of equations by a time marching method, c.f. Theorem \ref{errorest}.  Numerical evidence is given which indicates that  one approximates weak solutions in two dimensions.

Closely related to this paper are \cite{Brenner2010a, Brenner2010b,Neilan2013b}. Like the authors of these papers, we also use a fixed point argument but our approach is essentially different. No jump term is used in our variational formulation. We are able to give error estimates for Lagrange elements of degree $d \geq 2$ with no smoothness assumption on the boundary. 
This is achieved by a rescaling argument. The fixed point argument we use to establish the well-posedness of \eqref{m1h} also yields the theoretical convergence of the time marching iterative method. 

The use of the standard Lagrange finite element spaces in connection with the numerical resolution of \eqref{m1} also appears in mixed methods. A least squares formulation was used in \cite{GlowinskiICIAM07} and recently a direct mixed formulation was presented in \cite{Lakkis11}. The latter is essentially the limiting case of the mixed method for the vanishing moment methodology, c.f. \cite{Feng2009a} and the references therein. The vanishing moment methodology is a singular perturbation approach to the Monge-Amp\`ere  equation with the perturbation a multiple of the bilaplacian. The convergence and error estimates for the methods introduced in  \cite{GlowinskiICIAM07} are still open problems and mixed methods typically lead to large system of equations.

In view of having numerical results  for non smooth solutions, it is natural to use a time marching method, and not Newton's method, for solving the discrete nonlinear system of equations. 
The numerical experiments indicate that the method may be valid for the so-called viscosity solutions. This is a fascinating and challenging issue and its resolution involves additional new ideas different from the techniques for error analysis used in this paper. 
We wish to address this issue in a separate work \cite{Awanou-Std04}.

Our approach may be viewed as a variant of the method introduced in \cite{Brenner2010b}. 
As pointed out in  \cite{Brenner2010b}  a numerical method based on Lagrange elements and the formulation \eqref{m1h} does not work 
in theory in the sense
 it is difficult to use a fixed point argument which consists in linearization at the exact solution. The authors in \cite{Brenner2010b}  ingeniously added jump terms to facilitate the above approach. On the other hand, our numerical experiments indicate that the above approach works if the discrete nonlinear system of equations is solved by a time marching method. 
 An advantage of the time marching method is that the user only needs access to a Poisson solver to implement the scheme. The main advantage however is that one has numerical evidence of convergence for non-smooth solutions. 
 Obviously the time marching method can also be applied to the discretization proposed in  \cite{Brenner2010b} but we believe that in the context of non-smooth solutions the jump terms in the discretization proposed there may not be necessary. We have chosen not to treat curved boundaries for simplicity and to focus on the main ideas. The main motivation to assume that the domain is smooth and strictly convex is to guarantee the existence of a smooth solution for smooth data. One then faces the difficulty of practically imposing Dirichlet boundary conditions, a problem solved in  \cite{Brenner2010b} by the use of Nitsche method. Here instead we will make the assumption ubiquitous in finite element analysis of numerous problems that the solution is smooth on a polygonal domain. 

 We believe that the fixed point argument used in this paper and/or the strategy of rescaling the Monge-Amp\`ere equation would prove useful in resolving other outstanding issues about the numerical analysis of Monge-Amp\`ere type equations, see for example \cite{Awanou-Mixed-Time}.  For another example, our fixed point-rescaling argument provides an alternative to \cite{Neilan2013b} for the proof of  the well-posedness of the discretization proposed in \cite{Brenner2010b,Brenner2010a} for quadratic finite elements. Essentially, the rescaling argument is appropriate whenever an argument can be made that a result holds for the Monge-Amp\`ere equation provided the solution is sufficiently small. Thus, instead of describing the whole rescaling argument, one may simply prove results for the case when the exact solution is sufficiently small. 


In fact the results of this paper are similar to the ones announced in the context of $C^1$ conforming approximations in a technical report by the author \cite{AwanouPseudo10} but the analysis is more involved. Exploiting that similarity, pseudo transient continuation methods can be developed for \eqref{m1} by taking appropriate nonconforming discretizations of the iterative methods proposed in \cite{AwanouPseudo10}. We do not pursue this line of investigation in this paper. The properties of the Lagrange finite element spaces used in our analysis, namely an approximation property and inverse estimates, also hold for certain $C^1$ conforming approximations. 
Thus our error estimates hold for these as well. The error estimates hold for weaker assumptions on the exact solution, namely that $u \in W^{3,\infty}(K)$ on each element $K$, is strictly convex on each element and solves \eqref{m1h}. 

We organize the paper as follows. In section \ref{notation}, we give the notation used and recall some facts about determinants and Lagrange finite element spaces. The properties of the finite element spaces needed for our analysis are stated as well as the requirements on the exact solution. 
We prove existence and uniqueness of the discrete problem \eqref{m1h} with the convergence of the time marching method in section \ref{posedness}. 
In section \ref{numerical} we give the numerical results. 

\section{Notation and Preliminaries} \label{notation}
Let $\mathbb{P}_d$ denote the space of polynomials of degree less than or equal to $d$. We use the usual notation $L^p(\Omega), 2 \leq p \leq \infty$ for the Lebesgue spaces and $W^{s,p}(\Omega), 1 \leq s < \infty$ for the Sobolev spaces of elements of  $L^p(\Omega)$ with weak derivatives of order less than or equal to $s$ in $L^p(\Omega)$. The norms and semi-norms in $W^{s,p}(\Omega)$ are denoted by $||\, ||_{s,p}$ and $|\, |_{s,p}$ respectively and when $p=2$ we will simply use $||\, ||_{s}$ and $|\, |_{s}$. Thus the $L^p$ norm is denoted $||\, ||_{0}$. We will use the simpler notation $|| \ ||_{\infty}$ for the norm in $L^{\infty}(\Omega)$.

For a function defined on an element $K$ or more generally on a subdomain $S$, we will add $K$ or $S$ to the norm and semi-norm notation. We will need a broken Sobolev norm
$$
||v||_{s,p,\mathcal{T}_h} = \bigg( \sum_{K \in \mathcal{T}_h} ||v||_{s,p,K}^p\bigg)^{\frac{1}{p}},
$$
with the above conventions for the case when $p=2$.

For a matrix field $A$, we define $||A||_{\infty}=\max_{i,j=1,\ldots n} ||A_{ij}||_{\infty}$. 
We denote by $n$ the unit outward normal vector to $\partial \Omega$ and by $n_K$ the unit outward normal vector to $\partial K$ for an element $K$. 


For two matrices $A=(A_{ij})$ and $B=(B_{ij})$,  $A: B=\sum_{i,j=1}^n A_{ij} B_{ij}$ denotes their Frobenius inner product.  The divergence of a matrix field is understood as the vector obtained by taking the divergence of each row. We use the notation $D u$ to denote the gradient vector and for a matrix $A$, $\cof A$ denotes the matrix of cofactors of $A$.

A quantity which is constant is simply denoted by $C$. Throughout the paper, for a discrete function $v_h$, the Hessian $D^2 v_h$ is always computed element by element. We will assume that $0<h \leq 1$.

\subsection{Computations with determinants}


\begin{lemma} \label{MeanV}  For $u, v \in C^2(K)$ we have
$$
\det D^2 u - \det D^2 v = \cof(t D^2 u + (1-t) D^2 v): (D^2 u - D^2 v),
$$
for some $t \in [0,1]$.
\end{lemma}
\begin{proof}
The result follows from the mean value theorem and the expression of the derivative of the mapping $F: u \to \det D^2 u$. We have $F'(u)(v) = (\cof D^2 u):D^2 v$. First note that  $\partial (\det A)/(\partial A_{ij}) = (\cof A)_{ij}$. See for example formula (23) p. 440 of \cite{Evans1998}. The result then follows from the chain rule.
\end{proof}

\begin{lemma} \label{cofbound} 
For $n=2$ and $n=3$, and two matrix fields $\eta$ and $\tau$
\begin{align}
||\cof (\eta):\tau||_{0} & \leq C ||\eta||_{\infty}^{n-1} ||\tau||_{0}  \label{det-bound} \\
||\cof (\eta) - \cof (\tau) ||_{0} & \leq C ( || t \eta + (1-t)\tau||_{\infty})^{n-2}||\eta-\tau||_{0}. \label{cof-bound2}
\end{align}
\end{lemma}
\begin{proof}
The bound \eqref{det-bound} is  given by a direct computation.
For $n=2$, we have $\cof(\eta) - \cof(\tau)=  \cof(\eta-\tau)$ from which the result follows. For $n=3$ we use the mean value theorem. It is enough to estimate the first entry of 
$\cof(\eta) - \cof(\tau)$ which is equal to
\begin{align*}
\det \begin{pmatrix} \eta_{22} & \eta_{23} \\ \eta_{32} & \eta_{33} \end{pmatrix} - \det \begin{pmatrix} \tau_{22} & \tau_{23} \\ \tau_{32} & \tau_{33} \end{pmatrix} & =
\cof\bigg( t \begin{pmatrix} \eta_{22} & \eta_{23} \\ \eta_{32} & \eta_{33} \end{pmatrix}  +(1-t) \begin{pmatrix} \tau_{22} & \tau_{23} \\ \tau_{32} & \tau_{33} \end{pmatrix} \bigg): \\
& \begin{pmatrix} \eta_{22} - \tau_{22} &\eta_{23} - \tau_{23} \\\eta_{32} - \tau_{32} & \eta_{33} - \tau_{33} \end{pmatrix}, 
\end{align*}
for some $t \in [0,1]$. Direct computation then give 
\eqref{cof-bound2}. 
\end{proof}

\subsection{Assumptions on the approximation spaces}
For the discretization \eqref{m1h}, one can use either the Lagrange finite element spaces or certain finite dimensional spaces of $C^1$ functions. To make our results applicable to other types of discretizations, we formulate our assumptions on the approximation spaces.

\begin{ass} \label{opt-app} Approximation property. The finite dimensional space $V_h \subset H^1(\Omega)$ contains the Lagrange space of degree $d$ 
$$
\{ \, v_h \in C^0(\tir{\Omega}), v_h|_K \in \mathbb{P}_d, \forall K \in \mathcal{T}_h\, \},
$$
and there exists a linear interpolation operator
$I_h$ mapping $C^r(\tir{\Omega})$ for $r=0$ or $r=1$ into $V_h$ and a constant $C$
such that if $w$ is in the Sobolev space $W^{l+1,p}(\Omega), 1 \leq p \leq \infty$, $0 \leq l \leq d$
\begin{equation}
|| w -I_h w ||_{k,p,\mathcal{T}_h} \leq C_{ap} h^{l+1-k} |w|_{l+1,p}, \label{schum}
\end{equation}
for $k=0, 1, 2$. 
\end{ass}
The interpolant $g_h$ used in \eqref{m1h} is taken as $I_h$ applied to a continuous extension of $g$.

When $V_h$ is the Lagrange finite element space, the interpolant  $I_h$ is taken as the standard interpolation operator defined from the degrees of freedom. 
It is then known that Assumption \ref{opt-app} holds \cite{Brenner02}.

As a consequence of \eqref{schum},
\begin{equation}
||I_h w||_{k,p,\mathcal{T}_h} \leq (1+C_{ap}) ||w||_{k,p},   w \in W^{k,p}(\Omega),  k=0,1, 2, \label{stable}
\end{equation}
for all $p$. 

\begin{ass} \label{inverse} Inverse estimates
\begin{equation}
||w_h||_{t, p,\mathcal{T}_h} \leq C_{inv} h^{s-t+\min(0,\frac{n}{p}-\frac{n}{q})} ||w_h||_{s,q,\mathcal{T}_h},  \label{inverse}
\end{equation}
for $0 \leq s \leq t, 1 \leq p,q\leq \infty$ and  $w_h \in V_h$.
\end{ass}
The inverse estimates hold for the Lagrange finite element spaces as a consequence of the quasi-uniformity assumption on the triangulation \cite{Brenner02}.


\subsection{Assumptions on the exact solution}

Let $\lambda_1(A)$ and $\lambda_n(A)$ denote the smallest and largest eigenvalues of a symmetric matrix $A$. We make the following assumption on the exact solution:

\begin{ass} \label{psc} 
Local piecewise smooth and strict convexity assumption. The solution $u$ of \eqref{m1} is in $W^{3,\infty}(\mathcal{T}_h) \cap H^1(\Omega)$, strictly convex on each element $T$ and for constants $m', M' >0$, independent of $h$
$$
m' \leq \lambda_1(D^2 u(x)) \leq \lambda_n(D^2 u(x))  \leq M', \forall x \in K, K \in \mathcal{T}_h.
$$
Moreover, we require the exact solution $u$ to solve the problem: find $u \in W^{2,\infty}(\mathcal{T}_h)$, strictly convex on each element $T$, such that $u=g \ \text{on} \ \partial \Omega$ and 
\begin{align} \label{m2}
\sum_{K \in \mathcal{T}_h } \int_{K}  (\det D^2 u) v \ud x =
\int_{\Omega} f v \ud x, \forall v \in V_h \cap H^1_0(\Omega).
\end{align}
\end{ass}
We note that Assumption \ref{psc}  trivially holds for a strictly convex solution $u$ in $C^{3}(\tir{\Omega})$. 
In that case $ f \geq c_0>0$ for a constant $c_0$. The main motivation of Assumption \ref{psc} is our 
claim in \cite{Awanou-Std04} that the numerical approximation of the viscosity solution of \eqref{m1}  can be reduced to solving problems of this type. 

\section{Well-posedness of the discrete problem and error estimates} \label{posedness}
The proof of all lemmas in this section are given at the end of the section.

We first state a fundamental observation about the behavior of discrete functions near the interpolant $I_h u$.

\begin{lemma} \label{lem0}
There exists $\delta > 0$ such that for $h$ sufficiently small and for all $v_h \in V_h$ with $||v_h- I_h u||_1 < \delta/(2 C_{inv}) h^{1+n/2}$, $D^2 (v_h|_K)$ is positive definite with 
$$\frac{m'}{2} \leq \lambda_1D^2 (v_h|_K) \leq \lambda_nD^2 (v_h|_K) \leq \frac{3 M'}{2}$$ 
where $m'$ and $M'$ are the constants of Assumption \ref{psc}. Thus $\cof D^2 (v_h|_K)$ is invertible on each element $T$. 
\end{lemma}

Let
\begin{equation} \label{delta-h}
\delta_h = \frac{\delta}{2 C_{inv}} h^{1+\frac{n}{2}}.
\end{equation}
By Lemma \ref{lem0}, for $v_h \in V_h, ||v_h- I_h u||_1 \leq \delta_h$, $v_h$ is piecewise strictly convex with smallest eigenvalue bounded below by $m'/2$ and above by $3M'/2$. Put
$$X_h= \{ \, v_h \in V_h, v_h=g_h \, \text{on} \, \partial \Omega, ||v_h-I_h u||_1 < \delta_h \, \}.$$

As a consequence of Assumption \ref{psc}
\begin{lemma} \label{lem-1}
There exists constants $m, M >0$ independent of $h$ such that for all $v_h \in X_h$
$$
m \leq \lambda_1(\cof D^2 v_h(x)) \leq \lambda_n(\cof D^2 v_h(x))  \leq M, \forall x \in K, K \in \mathcal{T}_h.
$$ 
It follows that
\begin{equation} \label{p-d-K}
m |w|_{1,K}^2 \leq \int_{K} [(\cof\, D^2 v_h(x)) D w(x)] \cdot D w(x) \, \ud x \leq  M |w|_{1,K}^2, w \in H^1(K). 
\end{equation}
\end{lemma}

The main result of this section is the following theorem
\begin{thm} \label{errorest}
Let the finite dimensional spaces $V_h \subset H^1(\Omega)$ contain piecewise polynomials of degree $d \geq 2$. Assume that the spaces $V_h$ satisfy Assumption \ref{opt-app} 
of approximation property and Assumption \ref{inverse} of inverse estimates.
Assume also that the exact solution $u \in W^{l+1,\infty}(\mathcal{T}_h) \cap H^1(\Omega), 2\leq l \leq d$ satisfies Assumption \ref{psc} of strict convexity and solves \eqref{m2}. Then the problem \eqref{m1h} has a unique solution $u_h$ which is strictly convex on each element  and we have the error estimates
\begin{align*}
  ||u-u_h||_{2,\mathcal{T}_h } &\leq C h^{l-1}\\
  ||u-u_h||_1 &\leq C h^{l},
\end{align*} 
for $h$ sufficiently small. Moreover, with a sufficiently close initial guess $u_h^0$, the sequence defined by, $u_h^{k+1}=g_h$ on $\partial \Omega$
\begin{align}  \label{time-marching}
\begin{split}
 \frac{\nu}{\alpha^{n-1}} \int_{\Omega} D u_h^{k+1} \cdot D v_h \ud x & = \frac{\nu}{\alpha^{n-1}}  \int_{\Omega} D u_h^{k} \cdot D v_h \ud x 
  - \int_{\Omega} f v_h \ud x \\
  & \quad  \qquad \qquad \qquad +\sum_{K \in \mathcal{T}_h } \int_{K}  (\det D^2 u_h^k) v_h \ud x, 
\end{split}
\end{align}
$\forall v_h \in V_h \cap H_0^1(\Omega)$, converges linearly to $u_h$ in the $H^1$ norm for $\nu=(M+m)/2, \alpha = h^3$ and for $h$ sufficiently small.
\end{thm}
Before we give the proof of the above theorem we will state several lemmas whose proof are given at the end of the section. 

We recall that $\alpha >0$ is a small parameter which may depend on $h$.  For $\rho >0$, let
$$
B_h(\rho) = \{v_h \in V_h, v_h= g_h \, \text{on} \, \partial \Omega,  ||v_h- I_h u||_1 \leq \rho \},
$$
where we do not indicate the dependence of $B_h(\rho)$ on $\alpha$ for simplicity.
The ball $B_h(\rho)$ is nonempty as it contains $ I_h u$.

If $v_h \in B_h(\rho), ||\alpha v_h  -\alpha  I_h u||_1 \leq \alpha \rho$. Thus
\begin{align} \label{ball-inclusion}
\alpha B_h(\rho) \subset X_h \, \text{for} \, \alpha \rho \leq \frac{\delta_h}{2} \, \text{which holds for } \, \alpha = h^{3} \, \text{and} \, h \, \text{small enough}.
\end{align}

For a given $v_h \in V_h$, $v_h= g_h$ on $\partial \Omega$, define $T(\alpha v_h) \in V_h$ as the solution of 
\begin{align} \label{operatorB}
\begin{split}
\nu \int_{\Omega} D T(\alpha v_h) \cdot D w_h \ud x & = \nu  \int_{\Omega} D (\alpha v_h) \cdot D w_h \ud x \\
& \qquad   + \alpha^n \sum_{K \in \mathcal{T}_h } \int_{K}  (\det D^2 v_h) w_h \ud x \\
& \qquad \qquad  - \alpha^n \int_{\Omega} f w_h  \ud x, \forall w_h \in V_h \cap H_0^1(\Omega),
\end{split}
\end{align}
with $\alpha v_h-T(\alpha v_h) = 0$ on $\partial \Omega$ and we recall that $\nu=(M+m)/2$ where $M$ and  $m$ are the constants of Lemma \ref{lem-1}.

We will show that $T$ has a unique fixed point $\alpha u_h$ with $u_h$ in $B_h(\rho)$ for $h$ sufficiently small.

The motivation to introduce the damping parameter $\alpha$ is that it allows to solve a rescaled version of \eqref{m1}. Indeed $\det D^2 u=f$ is equivalent to $\det \alpha D^2 u= \alpha^n  f$. Taking $\alpha$ as a power of $h$ will play a crucial role in proving the well-posedness of \eqref{m1h}  and 
obtaining optimal error estimates. 

\begin{lemma} \label{lem1}
The mapping $T$ is well defined and if $\alpha u_h$ is a fixed point of $T$, i.e. $T(\alpha u_h)= \alpha u_h$, then $u_h$ solves \eqref{m1h}.
\end{lemma}
The next lemma says that the mapping $T$ does not move the center $I_h u$ of a ball $B_h(\rho)$ too far.

\begin{lemma} \label{lem2}
We have
\begin{equation}
||\alpha  I_h u -T(\alpha I_h u) ||_1 \leq  C_1 \alpha^n h^{l-1}  ||u ||_{2,\infty}^{n-1} || u||_{l+1}, \label{ballcenter}
\end{equation}
\end{lemma}

The next two lemmas establish the contraction mapping property of $T$ under the assumption that $d \geq 2$ and $\alpha=h^{3}$. 

\begin{lemma} \label{lem3}
For $h$ sufficiently small, and $\rho>0$, $T$ 
is a strict contraction mapping in the ball
$\alpha B_h(\rho)$,
i.e. for $v_h, w_h \in  B_h(\rho)$
$$
||T(\alpha v_h)-T(\alpha w_h)||_1 \leq a ||\alpha v_h-\alpha w_h||_1, 0<a< 1.
$$
\end{lemma}

\begin{lemma} \label{lem4}
For $h$ sufficiently small and $\rho=  1/(1-a) C_1 \alpha^{n-1} h^{l-1} || u||_{l+1,\infty}^n$ where $C_1$ is the constant in Lemma \ref{lem2}, $T$ 
is a strict contraction in $\alpha B_h(\rho)$ and maps $\alpha B_h(\rho)$ into itself.
\end{lemma}

The previous lemmas will readily allows us to conclude the solvability of \eqref{m1h} and derive error estimates in the $H^1$ norm by using the explicit expression of the radius $\rho$ of the above lemma. 
We can now give the proof of Theorem~\ref{errorest}.

\begin{proof}[Proof of Theorem~\ref{errorest}]
Since the mapping $T$ is a strict contraction which maps $\alpha B_h(\rho)$ into itself, the existence of a fixed point $\alpha u_h$ with $u_h \in B_h(\rho)$ follows from the Banach fixed point theorem. 
By Lemma \ref{lem1} $u_h$ solves \eqref{m1h}.

From the expression of $\rho$ given in Lemma \ref{lem4} we get using the value of $\alpha=h^3$
\begin{align*}
 ||u-u_h||_1& \leq  ||u-I_h u||_1 + ||I_h u - u_h||_1 \leq C h^{l} |u|_{l+1} + 2 C_1 \alpha^{n-1} h^{l-1}  ||u ||_{2,\infty}^{n-1} || u||_{l+1}\\
 & \leq C h^l + 2 C_1 h^{l+3n -4}  ||u ||_{2,\infty}^{n-1} || u||_{l+1} \leq C h^l ,
\end{align*}
which proves the $H^1$ error estimate.
By \eqref{schum} and \eqref{inverse},
\begin{align*}
  ||u-u_h||_{2,\mathcal{T}_h } & \leq   ||u-I_h u||_{2,\mathcal{T}_h } +   || I_h u - u_h||_{2,\mathcal{T}_h } \\
  & \leq  ||u-I_h u||_{2,\mathcal{T}_h } + h^{-1}  || I_h u - u_h||_1 \\
  & \leq C h^{l-1} |u|_{l+1} + C h^{l+3n -5}, 
\end{align*}
which proves that
$$
  ||u-u_h||_{2,\mathcal{T}_h }  \leq C h^{l-1} |u|_{l+1}.
$$
Finally we prove the convergence of the time marching method \eqref{time-marching}. Since $T$ is a strict contraction in $\alpha B_h(\rho)$, the sequence defined by $\alpha u_{k+1} = T(\alpha u_k)$, $u_{k+1}=u_k$ on $\partial \Omega$ converges linearly to $\alpha u_h$. Simplifying by $\alpha^n$, we get the convergence of \eqref{time-marching}.

\end{proof}

We conclude this section with the proofs of Lemmas \ref{lem0}--\ref{lem-1} and \ref{lem1}--\ref{lem4}.

\begin{proof}[Proof of Lemma~\ref{lem0}]
Recall that the eigenvalues of a (symmetric) matrix are continuous functions of its entries, as roots of the characteristic equation,  \cite{Ostrowski60} Appendix K, or \cite{Harris87}. Thus  for all $\epsilon > 0$, there exists $\delta > 0$ such that for $v \in W^{2,\infty}(\Omega)$, $|v-u|_{2,\infty} \leq \delta $
implies $|\lambda_1( D^2 v(x)) - \lambda_1( D^2 u(x))| < \epsilon$ a.e. in   $ \Omega$. 

By Assumption \ref{psc} $\lambda_1(  D^2 u(x)) \geq  m',$ a.e. in   $ \Omega$, and with $\epsilon=m'/2$, we get
$\lambda_1(D^2 v(x)) > m'/2, $ a.e. in   $ \Omega$. We conclude that for $|v-u|_{2,\infty} \leq \delta $, $\lambda_1( D^2 v(x)) > m'/2, $ a.e. in   $ \Omega$. 

Now, by \eqref{schum}, $|u-I_h u|_{2,\infty} \leq C_{ap} h^{d-1} |u|_{d+1,\infty}$. So for $h$ sufficiently small,  $|u-I_h u|_{2,\infty} \leq \delta/2$. Moreover by \eqref{inverse} and the assumption of the lemma
\begin{align*}
|v_h-I_h u|_{2,\infty} \leq C_{inv} h^{-1-\frac{n}{2}} ||v_h-I_h u||_1 \leq \frac{\delta}{2}.
\end{align*}
It follows that $\lambda_1( D^2 v_h(x)) > m'/2, $ a.e. in   $ \Omega$ as claimed. 

If necessary by taking $\delta$ smaller, we have  $|\lambda_n( D^2 v_h(x)) - \lambda_n( D^2 u(x))| < M'/2$ a.e. in   $ \Omega$. Thus
$\lambda_n( D^2 v_h(x)) \leq \lambda_n( D^2 u(x)) +M'/2 \leq 3 M'/2$. 
This concludes the proof.
\end{proof}

\begin{proof}[Proof of Lemma~\ref{lem-1}] We first note that by Lemma \ref{lem0}, there exists constants $m, M >0$ such that $
m \leq \lambda_1(\cof D^2 v_h(x)) \leq \lambda_n(\cof D^2 v_h(x))  \leq M $ a.e. in $\Omega$ for $v_h \in X_h$. 
To prove this, recall that for an invertible  matrix $A$,
$\cof A = (\det A) (A^{-1})^T$. Since a matrix and its transpose have the same set of eigenvalues, the eigenvalues of $\cof A$ are of the form $ \det A/\lambda_i$ where 
$\lambda_i, i=1,\ldots,n$ is an eigenvalue of $A$. Applying this observation to $A = D^2 u(x)$ and using Lemma \ref{lem0}, we obtain that the eigenvalues of
$\cof D^2 v_h(x)$ are a.e. uniformly bounded below by $m=(m')^{n}/M'$ and above by $M=(M')^{n}/m$.

Since $\lambda_1(D^2 v_h(x))$ and $\lambda_n(D^2 v_h(x))$ are the minimum
and maximum respectively of the Rayleigh quotient $[(\cof\, D^2 v_h(x)) z] \cdot z/||z||^2$, where $||z||$ denotes the standard Euclidean norm in $\R^n$, we have
$$
m' ||z||^2 \leq [(\cof\, D^2 v_h(x)) z] \cdot z \leq  M' ||z||^2, z \in \R^n. 
$$
This implies
\begin{equation*}
m |w|_{1,K}^2 \leq \int_{K} [(\cof\, D^2 v_h(x)) D w(x)] \cdot D w(x) \, \ud x \leq  M |w|_{1,K}^2, w \in H^1(K). 
\end{equation*}

\end{proof}

\begin{proof}[Proof of Lemma~\ref{lem1}]   
The existence of $T(\alpha v_h)$ solving \eqref{operatorB} is an immediate consequence of the Lax-Milgram lemma. 

If  $T(\alpha u_h) = \alpha u_h$, then 
\begin{align*} 
\alpha^n\sum_{K \in \mathcal{T}_h } \int_{K}  (\det D^2 u_h)  v_h \ud x = 
\alpha^n \int_{\Omega} f v_h \ud x, \forall v_h \in V_h \cap H^1_0(\Omega),
\end{align*}
and thus $u_h$ solves 
\eqref{m1h}. Conversely if $u_h$ solves \eqref{m1h},  $\alpha u_h$ is a fixed point of $T$.

\end{proof}

\begin{proof}[Proof of Lemma~\ref{lem2}]
From \eqref{schum} and \eqref{stable}, we obtain
\begin{align}
||u-I_h u||_2& \leq C h^{l-1} |u|_{l+1} \label{lem21} \\
 ||I_h u||_{2,K} & \leq C ||u||_{2 }. \label{lem22} 
\end{align}
Put $w_h = \alpha I_h u -T(\alpha I_h u)$ and note that $w_h \in H_0^1(\Omega)$. Since the exact solution solves \eqref{m2}, we have
\begin{align*}
\int_{\Omega} f w_h \ud x & =  \int_{\Omega}  (\det D^2 u)  w_h \ud x.
\end{align*}
With $v_h= I_h u$  in \eqref{operatorB}, we get
\begin{align} \label{partialB}
\begin{split}
\nu  \int_{\Omega} D [T(\alpha v_h)-\alpha v_h] \cdot D w_h \ud x& = \alpha^n \bigg( \sum_{K \in \mathcal{T}_h } \int_{K} (\det D^2 I_h u - \det D^2 u)  w_h \ud x \bigg). 
\end{split}
\end{align}
Put 
$$
z_h = \det D^2 I_h u - \det D^2 u.
$$
We have by Lemma \ref{MeanV} 
\begin{align*}
\begin{split}
z_h & = (\cof(t D^2 I_h u + (1-t) D^2 u)):(D^2 I_h u - D^2 u),
\end{split}
\end{align*}
for some $t \in [0,1]$.
Thus by Lemma \ref{cofbound}, \eqref{lem22}  and \eqref{lem21},
\begin{align*}
\begin{split}
||z_h||_{0,K} & \leq  C  || t D^2 I_h u + (1-t) D^2 u) ||_{\infty}^{n-1} ||D^2 I_h u -D^2 u||_{0,K} \\
& \leq C (||  I_h u||_{2, \infty} + ||u ||_{2,\infty})^{n-1} || I_h u - u||_{2,K} \\
&  \leq C  ||u ||_{2,\infty}^{n-1}h^{l-1} || u||_{l+1,K} \leq C h^{l-1}  ||u ||_{2,\infty}^{n-1} || u||_{l+1,K}.
\end{split}
\end{align*}
By Lemma \ref{lem1} and \eqref{partialB}, we get
\begin{align*}
\nu  |w_h|_1^2 & \leq C \alpha^n \sum_{K \in \mathcal{T}_h } ||z_h||_{0,K}  ||w_h||_{0,K}\\
&  \leq C \alpha^n h^{l-1}  ||u ||_{2,\infty}^{n-1} || u||_{l+1} ||w_h||_0 \leq   C \alpha^n h^{l-1}  ||u ||_{2,\infty}^{n-1} || u||_{l+1} ||w_h||_1.
\end{align*}
The result then follows by Poincare's inequality. 
\end{proof}

\begin{proof}[Proof of Lemma~\ref{lem3}] 
We define
$$
V_K = \{ \, v_h|_K, K \in \mathcal{T}_h, v_h \in X_h \, \},
$$
and denote by $V_K'$ the space of linear continuous functionals on $V_K$. For $F \in V_K'$, $||F||$ will denote the operator norm of $F$. We define a mapping $T_K: \alpha B_h(\rho)  \to V_K'$ defined by
\begin{align*}
\< T_K(\alpha v_h) , z_h \> & =  \alpha \int_{K} D v_h \cdot D z_h \ud x + \frac{\alpha^n}{\nu} \int_{K}  (\det D^2 v_h)  z_h \ud x  -  \frac{\alpha^n}{\nu}  \int_{K} f z_h  \ud x.
\end{align*}
Note that the restriction of elements of $\alpha B_h(\rho)$ to $K$ are in $V_K$.

{\bf Step 1}: We claim that for $v_h \in B_h(\rho)$ and $w_h \in V_K$, $||T_K'(\alpha v_h)(\alpha w_h)|| \leq  a ||w_h||_{1,K}$ for a constant $a$ such that $0 < a <1$  and 
$h$ sufficiently small.

\begin{align*} 
\begin{split}
\<T_K'(\alpha v_h)(\alpha w_h),z_h \> & =  \alpha \int_{K} D w_h \cdot D z_h \ud x + \frac{\alpha^n}{\nu} \int_{K}  [\div (\cof D^2 v_h) D w_h ]  z_h \ud x\\
& = \alpha \int_{K} D w_h \cdot D z_h \ud x - \frac{\alpha^n}{\nu} \int_{K}  [ (\cof D^2 v_h) D w_h ]  \cdot D z_h \ud x \\
& \qquad \qquad + \frac{\alpha^n}{\nu} \int_{\partial K}  z_h [ (\cof D^2 v_h) D w_h ]  \cdot n \ud s,
\end{split}
\end{align*}
and we used the expression of the derivative of the mapping $u \to \det D^2 u$ also used in the proof of Lemma \ref{MeanV}.
Therefore 
\begin{align} \label{TK-1}
\begin{split}
\<T_K'(\alpha v_h)(\alpha w_h),z_h \> & =  \alpha \int_{K} [ (I -  \frac{1}{\nu} \cof D^2  \alpha v_h  ) D w_h ]\cdot D z_h \ud x \\
& \qquad \qquad   + \frac{\alpha^n}{\nu} \int_{\partial K}  z_h [ (\cof D^2 v_h) D w_h ]  \cdot n \ud s,
\end{split}
\end{align}
where $I$ is the $n \times n$ identity matrix. We define
$$
\beta = \sup_{w_h \in V_K, ||w_h||_{1,K} = 1} \bigg|  \int_{K} [ (I -  \frac{1}{\nu} \cof D^2  \alpha v_h)  ) D w_h ]\cdot D w_h \ud x \bigg|.
$$
By \eqref{p-d-K} and \eqref{ball-inclusion}, we get
\begin{equation*}
(1-\frac{M}{\nu}) |w|_{1,K}^2 \leq \int_{K} [( I -\frac{1}{\nu}(\text{cof} \, D^2  \alpha v_h)) D w] \cdot D w \, \ud x \leq  (1-\frac{m}{\nu}) |w|_{1,K}^2, 
\end{equation*}
which gives by Poincare's inequality
\begin{equation*}
(1-\frac{M}{\nu}) \frac{ ||w||_{1,K}^2}{C_p^2} \leq \int_{K} [( I -\frac{1}{\nu}(\text{cof} \, D^2  \alpha v_h)) D w] \cdot D w \, \ud x \leq  (1-\frac{m}{\nu}) ||w||_{1,K}^2. 
\end{equation*}
Since $\nu =(M+m)/2$, $1-M/\nu= -(M-m)/(M+m)$ and $1-m/\nu=(M-m)/(M+m)$, we conclude that $\beta \leq \max \{ \, (M-m)/(M+m),  1/ C_p^2 (M-m)/(M+m) \, \}$. We may assume that $C_p \geq 1$ and thus
\begin{equation} \label{beta}
\beta \leq (M-m)/(M+m) < 1.
\end{equation}
Define $p_h= w_h/||w_h||_1$ and  $q_h= z_h/||z_h||_1$ for $w_h \neq 0$ and $v_h \neq 0$. Then 
\begin{align} \label{pp1}
\frac{\bigg|  \int_{K} [ (I -  \frac{1}{\nu} \cof D^2  \alpha v_h)  ) D w_h ]\cdot D z_h \ud x \bigg| }{||w_h||_1 ||z_h||_1}  =  \bigg|  \int_{K} [ (I -  \frac{1}{\nu} \cof D^2  \alpha v_h)  ) D p_h ]\cdot D q_h \ud x \bigg|.
\end{align}
We can define a bilinear form on $V_K$ by the formula
\begin{align*}
(p,q) & = \int_{\Omega} [(I -\frac{1}{\nu}(\text{cof} \, D^2  \alpha v_h)) D p] \cdot D q \, \ud x.
\end{align*}
Then because
$$
(p,q) = \frac{1}{4} ((p+q,p+q) - (p-q,p-q)),
$$
and using the definition of $\beta$, we get
$$
|(p_h,q_h)| \leq \frac{\beta}{4} ||p_h+q_h||_1^2 + \frac{\beta}{4}  ||p_h-q_h||_1^2 = \beta,
$$
since $p_h$ and $q_h$ are unit vectors in the $|| \, ||_1$ norm. It follows from \eqref{pp1} that
\begin{equation} \label{pp2}
\frac{\bigg|  \int_{K} [ (I -  \frac{1}{\nu} \cof D^2 \alpha  v_h)  ) D w_h ]\cdot D z_h \ud x \bigg| }{||w_h||_1 ||z_h||_1} \leq \beta.
\end{equation}
Next, we bound the second term on the right of \eqref{TK-1}. We need the scaled trace inequality
\begin{equation} \label{trace-inverse}
||v||_{L^2(\partial K) } \leq C h_K^{-\frac{1}{2}} ||v||_{L^2(K)} \, \forall v \in V_h.
\end{equation}
We have by Schwarz inequality and \eqref{trace-inverse}
\begin{align} \label{pp3}
\begin{split}
 \int_{\partial K}  z_h [ (\cof D^2 v_h) D w_h ]  \cdot n \ud s & \leq || (\cof D^2 w_h) D v_h ||_{0,K} ||z_h||_{0,K} \\
& \leq C  |w_h|_{2,\infty,K}^{n-1}  ||v_h||_{1,K}  ||z_h||_{0,K}\\
&\leq C h^{-(1+\frac{n}{2})(n-1)} ||w_h||_{1,K} ||v_h||_{1,K}  ||z_h||_{1,K}.
\end{split}
\end{align}
By \eqref{pp2} and  \eqref{pp3}, 
\begin{align*}
\frac{|\<T_K'(\alpha v_h)(\alpha w_h),z_h \> |}{||w_h||_{1,K}  ||z_h||_{1,K}} \leq \alpha(\beta + C \alpha^{n-1 } h^{-(1+\frac{n}{2})(n-1)} ||v_h||_{1,K} ).
\end{align*}

We conclude using the expression of  $\alpha=h^3$ and assuming $\rho \leq 1$ that
\begin{align*}
||T_K'(\alpha v_h)(\alpha w_h)|| & = \sup_{z_h \neq 0} \frac{|\<T_K'(\alpha v_h)(\alpha w_h),z_h \> |}{||z_h||_{1,K}} \\
& \leq \alpha (\beta + C h^{(2-\frac{n}{2})(n-1)} ||v_h||_{1,K} ) ||w_h||_{1,K}\\
& \leq \alpha  (\beta + C h^{(2-\frac{n}{2})(n-1)}  ||v_h - I_h u||_{1,K} + C  h^{(2-\frac{n}{2})(n-1)} ||I_h u||_{1,K} ) ||w_h||_{1,K} \\
&  \leq \alpha (\beta + C h^{(2-\frac{n}{2})(n-1)} \rho + C h^{(2-\frac{n}{2})(n-1)} || u||_{1} ) ||w_h||_{1,K} \\
&  \leq  (\beta + C h^{\frac{1}{2}} \rho + C   h^{\frac{1}{2}} || u||_{1}  ) || \alpha w_h||_{1,K},
\end{align*}
and we recall that $n=2,3$ allowing us to treat the two cases in a unifying fashion.

Since $\beta <1$, for $h$ sufficiently small $a=\beta + C h^{\frac{1}{2}} \rho + C   h^{\frac{1}{2}} || u||_{1}  <1$. 
This proves the result.

{\bf Step 2}:  The mapping $T_K$ is a strict contraction, i.e. for $v_h, w_h \in B_h(\rho)$, $||T_K(\alpha v_h) - T_K(\alpha w_h)|| \leq  a ||\alpha v_h-\alpha w_h||_{1,K}$, $0<a <1$. 

Using the mean value theorem
\begin{align*}
||T_K(\alpha v_h)-T_K(\alpha w_h)|| & = ||\int_0^1 T_K'(\alpha v_h +t(\alpha w_h-\alpha v_h))(\alpha w_h-\alpha v_h) \ud t || \\
& \leq \int_0^1 ||T_K'(\alpha v_h +t(\alpha w_h-\alpha v_h))(\alpha w_h-\alpha v_h)|| \ud t.
\end{align*}
Since $B_h(\rho)$ is convex, $v_h +t(w_h-v_h) \in B_h(\rho), t \in [0,1]$, and by the result established in step 1,
\begin{align*}
||T_K(\alpha v_h)-T_K(\alpha w_h)|| & \leq \int_0^1a ||\alpha w_h-\alpha v_h||_{1,K} \ud t =a ||\alpha w_h- \alpha v_h||_{1,K}.
\end{align*}

{\bf Step 3}: The mapping $T$ is a strict contraction in $\alpha B_h(\rho)$.

\begin{align*}
 \int_{\Omega} D(T(\alpha v_h) - T(\alpha w_h)) \cdot D \psi_h \ud x & =  \alpha \int_{\Omega} D(v_h - w_h) \cdot D \psi_h \ud x \\
&  \ + \frac{\alpha^n}{\nu}  \sum_{K \in \mathcal{T}_h }\int_{K}  ( \det D^2 v_h - \det D^2 w_h) \psi_h \ud x\\
& = \sum_{K \in \mathcal{T}_h } \< T_K(\alpha v_h) - T_K(\alpha w_h),\psi_h\>.
\end{align*}
With $\psi_h = T(\alpha v_h) - T(\alpha w_h)$, we obtain using the result from step 2.
\begin{align*}
 |T(\alpha v_h) - T(\alpha w_h)|_1^2 & \leq \sum_{K \in \mathcal{T}_h } || T_K(\alpha v_h) - T_K(\alpha w_h) || ||\psi_h||_{1,K} \\
& \leq a C_p  \sum_{K \in \mathcal{T}_h } ||\alpha v_h - \alpha w_h ||_{1,K} |\psi_h|_{1,K} \\
& \leq a C_p ||\alpha v_h-\alpha w_h||_1 |\psi_h|_1,
\end{align*}
where $C_p$ is the constant in the Poincare's inequality.
It follows that $||T(\alpha v_h) - T(\alpha w_h)||_1 \leq  a ||\alpha v_h-\alpha w_h||_1$.

\end{proof}

\begin{proof}[Proof of Lemma~\ref{lem4}]
Let $v_h \in B_h(\rho)$. Then by Lemma \ref{lem2}
\begin{align*}
||T(\alpha v_h) - \alpha I_h u ||_1 & \leq ||T(\alpha v_h) - T(\alpha I_h u)||_1 +  ||T(\alpha I_h u)-\alpha I_h u||_1 \\
& \leq a ||\alpha v_h - \alpha I_h u ||_1 +  C_1\alpha^n  h^{l-1} || u||_{l+1,\infty}^n \\
& = a ||\alpha v_h - \alpha I_h u ||_1 + (1-a)  \alpha \rho\\
& \leq a \alpha \rho + (1-a)    \alpha \rho \\
& \leq \alpha \rho,
\end{align*}
and we conclude that
$$
||T(\alpha v_h) - \alpha I_h u ||_1 \leq \alpha \rho.
$$
This proves the result.
\end{proof}

\begin{rem}
Let us assume that \eqref{m1h} has a strictly convex solution $u_h$ (independently of the smoothness of $u$). If in addition, its eigenvalues are bounded below and above by constants independent of $h$, then using again the continuity of the eigenvalues of a matrix as a function of its entries, we obtain the existence of $\delta'>0$ such that for $v_h$ in
$$
Y^h = \{ \, v_h \in V_h, v_h=g_h \, \text{on} \, \partial \Omega, ||v_h- u_h||_1 < C \delta' h^{1+\frac{n}{2}} \, \},
$$
$v_h$ is convex. It is not difficult to see that the mapping $T$ is also a strict contraction in $Y^h$ for $h$ sufficiently small. One obtains the linear  convergence of the iterative method \eqref{time-marching} to $u_h$ as follows:
\begin{align*}
||\alpha u^{k+1}_h - \alpha u_h||_1 & = ||T(\alpha u^{k}_h) - T(\alpha u_h)||_1 \leq a ||\alpha u^{k}_h - \alpha u_h||_1, 0< a <1. 
\end{align*}
Simplifying by $\alpha$ proves the claim.
\end{rem}

\section{Numerical Results} \label{numerical}
The implementation is done in Matlab. The computational domain is the unit square $[0,1]^2$
which is first divided into squares of side length $h$. Then each square is 
divided into two triangles by the diagonal with positive slope. We use standard test functions for numerical convergence to viscosity solutions of non degenerate Monge-Amp\`ere equations, i.e. for $f >0$ in $\Omega$.

Test 1: $u(x,y)=e^{(x^2+y^2)/2}$ with corresponding $f$ and $g$. This solution is infinitely differentiable.

Test 2:  $u(x,y)=-\sqrt{2-x^2-y^2}$ with corresponding $f$ and $g$. This solution is not in $H^2(\Omega)$.

Test 3: $g(x,y)=0$ and $f(x,y)=1$. No exact solution is known in this case.


For the test function in Test 1 which is a smooth function and the one in Test 3, we used the iterative method of Theorem \ref{errorest} with $\alpha=1$. For the non smooth solution of Test 2, we found the following truncated version more efficient. For $m=1,2, \ldots$, we consider truncating functions
$\chi_m(x)$ defined by $\chi_m(x)=-m $ for $x < -m $, $\chi_m(x)=x $ for $-m \leq x \leq m$ and $\chi_m(x)=m $ for $x > m $
and the sequence of problems
\begin{equation*}
\nu \int_{\Omega}  D u^{k+1,m}_h \cdot D v_h \ud x =  \nu \int_{\Omega}  D u^{k,m}_h \cdot D v_h \ud x + \sum_{K \in \mathcal{T}_h }  \int_{K} \chi_m     (\det D^2 u^{k,m}_h - f) v_h \ud x, 
\end{equation*}
with $ \ u^{k+1,m}_h=g_h \ \mathrm{on} \ \partial \Omega$. Also we use a broken $H^1$ norm in Table \ref{tab2}.

Compared with $C^1$ conforming approximations or mixed methods, the standard finite element method is less able to capture convex solutions. However we note the unusual high order convergence rate in the $L^2$ norm for the non smooth solution of Test 2. The optimal convergence rate of Theorem \ref{errorest} is an asymptotic convergence rate. 
For higher order elements, better numerical convergence rates are obtained with the iterative methods discussed in \cite{Awanou2009a}. In summary the method proposed in this paper is efficient for non smooth solutions and quadratic elements.

\begin{table}
 \begin{tabular}{c|cccc} 
 \multicolumn{4}{c}{}\\
$h$ & $||u-u_h||_{L^2}$ &  rate &  $||u-u_h||_{H^1}$&  rate  \\ \hline
$1/2$ &4.38 $10^{-2}$ & & 2.05 $10^{-1}$&   \\  
$1/4$ & 2.18 $10^{-2}$& 1.00& 1.04 $10^{-1}$&0.98  \\  
$1/8$ & 9.00 $10^{-3}$ & 1.28& 4.19 $10^{-2}$&1.31\\ 
$1/16$ & 2.76 $10^{-3}$ & 1.70& 1.28 $10^{-2}$& 1.71\\ 
$1/32$ & 7.35 $10^{-4}$ & 1.91& 3.40 $10^{-3}$&1.91 \\ 
$1/64$ & 1.86 $10^{-4}$ &1.98 & 8.65 $10^{-4}$& 1.97 \\ 
\end{tabular} 
\bigskip
\caption{Test 1 $d=2, \nu=50$} \label{tab1}
\end{table}

\begin{table}
 \begin{tabular}{c|cccc} 
 \multicolumn{4}{c}{}\\
$h$ & $||u-u_h||_{L^2}$ &  rate &  $||u-u_h||_{H^1_h}$&  rate  \\ \hline
$1/16$ &    1.79 $10^{-1}$ &  &  1.1718 & \\ 
$1/32$ &   6.54 $10^{-2}$ &1.45 &  5.47 $10^{-1}$ & 1.10 \\ 
$1/64$ &   1.24 $10^{-2}$ & 2.40 &  1.52 $10^{-1}$ &1.85 \\ 
$1/128$ &   2.10 $10^{-3}$ &2.56 &  6.00 $10^{-2}$ & 1.34 \\ 
$1/256$ &   4.91 $10^{-4}$ & 2.09&  4.27 $10^{-2}$ & 0.49 \\ 
$1/512$ &   1.29 $10^{-4}$ & 1.93 &  3.34 $10^{-2}$ & 0.35\\ 
\end{tabular} 
\bigskip
\caption{Test 2 $d=2, \nu=150, m=250$} \label{tab2}
\end{table}

\begin{figure}[tbp]
\begin{center}
\includegraphics[angle=0, height=4.5cm]{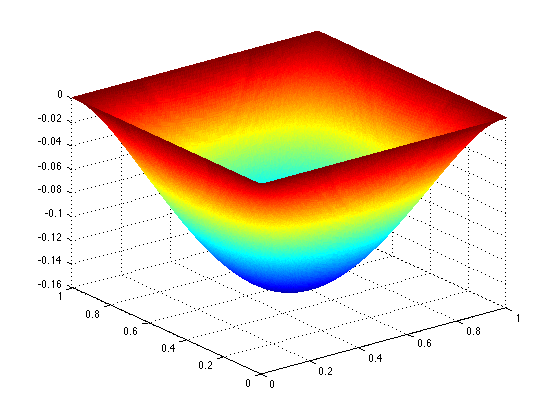}
\end{center}
\caption{Test 3 $d=2, h=1/2^7, \nu=50$
} \label{fig1}
\end{figure}

\end{document}